\documentclass[11pt,letterpaper]{amsart}
\usepackage{graphicx}
\usepackage{amsmath} 
\usepackage{amsthm,amsfonts,amssymb,mathrsfs,amscd,amstext,amsbsy} 
\usepackage{epic,eepic} 
\usepackage{yfonts}
\usepackage{paralist,enumerate}
\usepackage[all]{xy}
\usepackage{hyperref}
\hypersetup{colorlinks}

\newtheorem{theorem}{Theorem}[section]
\newtheorem{cor}[theorem]{Corollary}
\newtheorem{lemma}[theorem]{Lemma}

\newtheorem{theo}[theorem]{Theorem}
\newtheorem{lem}[theorem]{Lemma}
\newtheorem{pro}[theorem]{Proposition}

\newtheorem{exa}[theorem]{Example}

\newtheorem{Definition}[theorem]{Definition}
\newtheorem*{Definition*}{Definition}
\def\qed{\hfill \ifhmode\unskip\nobreak\fi\quad\ifmmode\Box\else$\Box$\fi\\ }
\begin{document}
\title[Circle action on oriented 6-manifold with isolated fixed points]{Circle actions on six dimensional oriented manifolds with isolated fixed points}
\author{Donghoon Jang}
\thanks{Donghoon Jang was supported by the National Research Foundation of Korea(NRF) grant funded by the Korea government(MSIT) (2021R1C1C1004158).}
\address{Department of Mathematics, Pusan National University, Pusan, South Korea}
\email{donghoonjang@pusan.ac.kr}
\begin{abstract}
To classify a group action on a manifold, the data associated with the fixed point set is essential. In this paper, we classify the fixed point data of a circle action on a 6-dimensional compact connected oriented manifold with isolated fixed points, where the fixed point data consists of the collection of signs and weights at the fixed points. We show that this fixed point data can be reduced to the empty collection by performing a sequence of operations. Specifically, we prove that one can successively take equivariant connected sums at fixed points with $S^6$, $\mathbb{CP}^3$, or 6-dimensional analogues of the Hirzebruch surfaces (and their oppositely oriented counterparts), resulting in a fixed-point-free action on a compact connected oriented 6-manifold.
\end{abstract}

\maketitle

\section{Introduction}

Torus actions on compact manifolds have been studied in low dimensions. In dimensions 1 and 2 the classification results are simple. Torus actions and circle actions on 3- and 4-manifolds were studied and classified in 1960's and 1970's. Raymond classified circle actions on compact 3-manifolds \cite{R}.
Orlik and Raymond proved that a 2-torus action on a simply connected closed orientable 4-manifold is $S^4$, or a connected sum of $\mathbb{CP}^2$, $\overline{\mathbb{CP}^2}$, and the Hirzebruch surfaces \cite{OR1}. Fintushel proved an analogous result that a circle action on a simply connected closed oriented 4-manifold is a connected sum of $S^4$, $\mathbb{CP}^2$, $\overline{\mathbb{CP}^2}$, and $S^2 \times S^2$ \cite{F2}. 
In the non-simply connected case, Orlik and Raymond showed that a 2-torus action on a closed orientable 4-manifold is determined by its orbit data \cite{OR2}. 
Fintushel similarly showed that a circle action on a closed oriented 4-manifold is determined by its orbit data \cite{F3}.
See also \cite{F1,J1,P} for classification results on certain 4-dimensional oriented $S^1$-manifolds.

Circle actions on different types of 4-manifolds with fixed points have also been studied. Carrell, Howard, and Kosniowski studied complex surfaces \cite{CHK}. Ahara and Hattori \cite{AH}, Audin \cite{Au}, and Karshon \cite{Ka} studied symplectic 4-manifolds. 
For both complex 4-manifolds and symplectic 4-manifolds, a common phenomenon arises: one can successively blow down any such manifold to a minimal one, which is $\mathbb{CP}^2$, a Hirzebruch surface, or a ruled surface. See also an extension of these results to almost complex 4-manifolds with finitely many fixed points \cite{J3}.

We consider a torus action on a manifold with a non-empty, finite fixed point set. Because the dimension of a manifold and the dimension of its fixed point set have the same parity, the next case to consider is dimension 6. 
For any type of manifold, classifying torus actions on 6-manifolds is more difficult, and known results require additional restrictions on the actions and/or the manifolds. For instance, McGavran and Oh considered $T^3$-actions on closed orientable 6-manifolds whose orbit space is $D^3$ \cite{MO}, Kuroki considered $T^3$-actions on closed oriented 6-manifolds with fixed points and vanishing odd cohomology groups \cite{Ku}, Tolman considered Hamiltonian $S^1$-actions on symplectic manifolds with minimal cohomology groups \cite{T} (see also \cite{L}), and Ahara \cite{Ah} and the author \cite{J2} considered $S^1$-actions on almost complex 6-manifolds with a small number of fixed points.

In this paper, we study circle actions on 6-dimensional compact oriented manifolds with isolated fixed points, without imposing any additional assumptions. A geometric formulation of our result is as follows.

\begin{theorem}\label{t11}
Let the circle group $S^1$ act on a 6-dimensional compact connected oriented manifold $M$ with a discrete fixed point set. 
Then, by successively taking equivariant connected sums at fixed points of $M$ and at fixed points of circle actions on $S^6$, $\mathbb{CP}^3$, $Z_1$, and $Z_2$ (and these with opposite orientations), one can construct another 6-dimensional compact connected oriented manifold equipped with a fixed-point-free circle action.
This process terminates after finitely many steps.
The circle actions on $S^6$, $\mathbb{CP}^3$, $Z_1$, and $Z_2$ all have non-empty finite fixed point sets.
\end{theorem}

For the manifolds $Z_n$, which are the 6-dimensional analogue of the Hirzebruch surfaces, see Example~\ref{e3}.

Let the circle group $S^1$ act on a $2n$-dimensional compact oriented manifold $M$ with a discrete fixed point set. Let $p$ be an isolated fixed point. The tangent space at $p$ decomposes into real 2-dimensional irreducible $S^1$-equivariant vector spaces
\begin{center}
$\displaystyle T_pM=\bigoplus_{i=1}^n L_{p,i}$,
\end{center}
where each $L_{p,i}$ is isomorphic to a complex 1-dimensional $S^1$-equivariant vector space, on which the circle acts as multiplication by $g^{w_{p,i}}$ for all $g \in S^1 \subset \mathbb{C}$, where $w_{p,i}$ is a non-zero integer. For each $i$, we choose an orientation of $L_{p,i}$ so that $w_{p,i}$ is positive. We call the positive integers $w_{p,i}$ the \textbf{weights} at $p$. Let $\epsilon_M(p)=+1$ if the orientation on $M$ agrees with the orientation on the representation space $L_{p,1} \oplus \cdots \oplus L_{p,n}$, and $\epsilon_M(p)=-1$ otherwise. We call $\epsilon_M(p)$ the \textbf{sign} of $p$. 
When there is no ambiguity, we omit the subscript and write $\epsilon(p)$ instead.
We define the \textbf{fixed point data} of $p$ to be 
\begin{center}
$\Sigma_p:=\{\epsilon(p),w_{p,1},\cdots,w_{p,n}\}$. 
\end{center}
We will always list the sign of $p$ first, followed by the weights. 
More precisely, the fixed point data of $p$ can be represented as the ordered pair $(\epsilon(p),\{w_{p,1},\cdots,w_{p,n}\})$, where $\{w_{p,1},\cdots,w_{p,n}\}$ is the multiset of the weights at $p$.
For notational simplicity, we write $\{\epsilon(p),w_{p,1},\cdots,w_{p,n}\}$ instead of the ordered pair.
When writing $\epsilon(p)$ inside the fixed point data, we omit the value 1 and only indicate the sign.
Finally, we define the \textbf{fixed point data} of $M$ to be the collection 
\begin{center}
$\displaystyle \Sigma_M:=\bigcup_{p \in M^{S^1}} \{\epsilon(p),w_{p,1},\cdots,w_{p,n}\},$
\end{center}
which consists of the fixed point data for all fixed points of $M$.

For the classification of a torus action on a manifold, the fixed point data is essential. A combinatorial version of our result provides a classification of the fixed point data of any 6-dimensional compact, oriented $S^1$-manifold with isolated fixed points.

\begin{theorem}\label{t12}
Let the circle group $S^1$ act on a 6-dimensional compact connected oriented manifold $M$ with a discrete fixed point set. 
Then we can successively apply a combination of five types of operations to reduce the fixed point data of $M$ to the empty collection. There is a definite procedure that terminates in finitely many steps.

More precisely, given the fixed point data $\Sigma_M$ of $M$, we can apply a combination of the following operations to reduce $\Sigma_M$ to the empty collection.

\begin{enumerate}
\item[(1)] Remove $\{+,A,B,C\}$ and $\{-,A,B,C\}$ together.
\item[(2)] Remove $\{\pm,A,B,C\}$ and $\{\mp,C-A,C-B,C\}$ where $A<B<C$, and add $\{\pm,A,B-A,C-A \}$ and $\{\mp, B,B-A,C-B\}$.
\item[(3)] Remove $\{\pm,A,B,C\}$ and $\{\pm,A,C-B,C\}$ where $A<B<C$, and add $\{\pm, C-B,C-A,A\}$, $\{\pm, C-B,B,A\}$, $\{\mp, C-B,B-A,A\}$, and $\{\pm, C-A,B-A,A\}$.
\item[(3')] Remove $\{\pm,A,B,C\}$ and $\{\pm,A,C-B,C\}$ where $B<A<C$, and add $\{\pm, C-B,C-A,A\}$, $\{\pm, C-B,B,A\}$, $\{\pm, C-B,A-B,A\}$, and $\{\mp, C-A,A-B,A\}$.
\item[(4)] Remove $\{\pm,A,A,C\}$ and $\{\pm,A,C-A,C\}$ where $2A<C$, and add $\{\pm,C-A,C-2A,A\}$, $\{\pm,C-A,A,A\}$, $\{\pm,C-A,A,A\}$, $\{\mp,C-2A,A,A\}$.
\item[(4')] Remove $\{\pm,A,A,C\}$ and $\{\pm,A,C-A,C\}$ where $A<C<2A$, and add $\{\mp,C-A,2A-C,A\}$, $\{\pm,C-A,A,A\}$, $\{\pm,C-A,A,A\}$, $\{\pm,2A-C,A,A\}$.
\item[(5)] Remove $\{\pm, C, A, A\}$ and $\{\mp, C, C-A, C-A\}$ where $2A<C$, and add $\{\pm,C-A,C-2A,A\}$, $\{\pm,C-A, A, A\}$, $\{\pm,C-A, A, A\}$, $\{\mp,C-2A, A, A\}$, $\{\pm, A, C-2A, C-A\}$, $\{\mp,A,C-A,C-A\}$, $\{\mp,A,C-A,C-A\}$, $\{\mp, C-2A, C-A, C-A\}$.
\end{enumerate}
\end{theorem}

The classification results of this paper may provide useful insights for classifying circle (and torus) actions on other types of 6-manifolds, such as (almost) complex and symplectic manifolds. For example, if one aims to classify an action on an (almost) complex 6-manifold $M$ with a finite fixed point set, one can verify whether the complex weights at the fixed points satisfy the conclusion of Theorem \ref{t12} when translated into the fixed point data of $M$ as an oriented manifold.

Theorem \ref{t12} implies that, given any such manifold $M$, we must be able to perform one of the operations described in Theorem \ref{t12} on the fixed point data of $M$. After applying one such operation, we can again perform another (unless we have already obtained the empty collection), and so on, until we arrive at the empty collection. The sign convention in Theorem \ref{t12} has the following implication: suppose the fixed point data of $M$ includes $\{-,A,B,C\}$ and $\{+,C-A,C-B,C\}$. Then we may apply Operation (2) to remove these, replacing them with $\{-,A,B-A,C-A\}$ and $\{+,B,B-A,C-B\}$. If the fixed point data of $M$ includes $\{+,A,B,C\}$ and $\{-,C-A,C-B,C\}$ (i.e., the same data as above but with opposite signs), then we can again apply Operation (2), this time removing these and adding $\{+,A,B-A,C-A\}$ and $\{-,B,B-A,C-B\}$.

The operations in Theorem \ref{t12} are designed to remove the largest weight $C$; every weight introduced in the new fixed point data is strictly smaller than $C$. Therefore, the process described in Theorem \ref{t12} terminates after finitely many steps.
Operation (1) corresponds to a connected sum with $S^6$; Operation (2) corresponds to a connected sum with $\mathbb{CP}^3$ (or $\overline{\mathbb{CP}^3}$); Operations (3) and (3') correspond to a connected sum with the manifold $Z_1$ (or $\overline{Z_1}$); Operations (4) and (4') correspond to a connected sum with the manifold $Z_2$ (or $\overline{Z_2}$); and Operation (5) corresponds to a connected sum with the manifold $Z_2 \sharp \overline{Z_2}$ from Example \ref{e8}, which is a connected sum at fixed points of two copies of $Z_2$, one with reversed orientation.
Unlike in standard conventions, we take each connected sum at two fixed points of $M$ and at two fixed points of an $S^1$-action on another manifold.
The only difference between Operations (3) and (3') is a sign issue that arises at two fixed points. In Operation (3), the fixed point data involved are $\{\mp, C-B, B-A, A\}$ and $\{\pm, C-A, B-A, A\}$, while in Operation (3'), they are $\{\pm, C-B, A-B, A\}$ and $\{\mp, C-A, A-B, A\}$. In the former case, $A < B$, and in the latter, $B < A$. Notably, if we allow weights to be negative, then the fixed point data $\{\pm, C-B, B-A, A\}$ in the former case becomes $\{\mp, C-B, A-B, A\}$ in the latter.
Similarly, the difference between Operations (4) and (4') is also due solely to such a sign issue at two fixed points.

If a 6-dimensional compact oriented $S^1$-manifold has exactly two fixed points, then the two fixed points have the same multiset of weights and opposite signs (see Theorem \ref{t28}). By performing Operation (1) of Theorem \ref{t12} once on the fixed point data of the manifold, we obtain the empty collection.

We illustrate how Theorem \ref{t12} works with an example. 

\begin{exa}
Let the circle act on $\mathbb{CP}^3$ by
\begin{center}
$g \cdot [z_0:z_1:z_2:z_3]=[z_0:g^a z_1:g^b z_2:g^c z_3]$
\end{center}
for some positive integers $a<b<c$, for all $g \in S^1 \subset \mathbb{C}$ and all $[z_0:z_1:z_2:z_3] \in \mathbb{CP}^3$. This action has four fixed points $[1:0:0:0]$, $[0:1:0:0]$, $[0:0:1:0]$, and $[0:0:0:1]$. The fixed point data of $\mathbb{CP}^3$ is
\begin{center}
$\{+,a,b,c\}$, $\{-,a,b-a,c-a\}$, $\{+,b,b-a,c-b\}$, $\{-,c,c-a,c-b\}$,
\end{center}
see Example \ref{e2}. Because the fixed points $[1:0:0:0]$ and $[0:0:0:1]$ have the largest weight $c$, we perform Operation (2) (which corresponds to taking an equivariant connected sum with $\overline{\mathbb{CP}^3}$ at the two fixed points $[1:0:0:0]$ and $[0:0:0:1]$ for both manifolds). This removes $\{+,a,b,c\}$ and $\{-,c-a,c-b,c\}$ and adds $\{+,a,b-a,c-a\}$ and $\{-,b,b-a,c-b\}$, resulting in the collection
\begin{center}
$\{-,a,b-a,c-a\}$, $\{+,b,b-a,c-b\}$, $\{+,a,b-a,c-a\}$, $\{-,b,b-a,c-b\}$.
\end{center}
Next, we perform Operation (1) twice on this collection: first to remove $\{-,a,b-a,c-a\}$ and $\{+,a,b-a,c-a\}$, and second to remove $\{+,b,b-a,c-b\}$ and $\{-,b,b-a,c-b\}$, thus reaching the empty collection. 
\end{exa}

Of course, for the action on $\mathbb{CP}^3$, one can take an equivariant connected sum with the same action on $\overline{\mathbb{CP}^3}$ at all fixed points to obtain a fixed-point-free $S^1$-manifold. However, for a general 6-dimensional compact oriented $S^1$-manifold $M$ with a discrete fixed point set, the situation is different: even if the fixed point data of $M$ contains $\{+,a,b,c\}$ and $\{-,c,c-a,c-b\}$, it need not contain $\{+,b,b-a,c-b\}$ and $\{+,a,b-a,c-a\}$.

The structure of this paper is as follows. In Section \ref{s2}, we review the necessary background on oriented $S^1$-manifolds with discrete fixed points. In Section \ref{s3}, we describe the $S^1$-actions on $S^6$, $\mathbb{CP}^3$, $Z_n$, and $Z_2 \sharp \overline{Z_2}$ that are needed to prove Theorems \ref{t11} and \ref{t12}. In Section \ref{s4}, we establish relationships between the fixed point data of fixed points in an isotropy submanifold. Finally, in Section \ref{s5}, we prove Theorems \ref{t11} and \ref{t12}.

\section*{Acknowledgements}

The author would like to thank Mikiya Masuda and Michael Wiemeler for fruitful discussions on the orientability of isotropy submanifolds that appear in Lemmas~\ref{codim-even} and \ref{dim2}.

\section{Background} \label{s2}

For an action of a group $G$ on a manifold $M$, we denote its fixed point set by $M^G$, that is,
\begin{center}
$M^G=\{m \in M \mid g \cdot m=m \textrm{ for all }g \in G\}$.
\end{center}
If $H$ is a subgroup of $G$, then $H$ acts on $M$ as a subgroup of $G$, and we denote its fixed point set by $M^H$.

Let the circle group $S^1$ act on a manifold $M$. The \textbf{equivariant cohomology} of $M$ is 
\begin{center}
$\displaystyle H_{S^{1}}^{*}(M)=H^{*}(M \times_{S^{1}} S^{\infty})$. 
\end{center}
Suppose that $M$ is oriented and compact. The projection map $\pi : M \times_{S^{1}} S^{\infty} \rightarrow \mathbb{CP}^{\infty}$ induces a natural push-forward map
\begin{center}
$\displaystyle \pi_{*} : H_{S^{1}}^{i}(M;\mathbb{Z}) \rightarrow H^{i-\dim M}(\mathbb{CP}^{\infty};\mathbb{Z})$
\end{center}
for all $i \in \mathbb{Z}$. This map is given by integration over the fiber $M$, and also denoted by $\int_M$.

\begin{theo} \emph{[ABBV localization theorem]} \cite{AB, BV} \label{t21}
Let the circle group $S^1$ act on a compact oriented manifold $M$. Let $ \alpha \in H_{S^{1}}^{\ast}(M;\mathbb{Q})$. As an element of $\mathbb{Q}(t)$,
\begin{center}
$\displaystyle \int_{M} \alpha = \sum_{F \subset M^{S^{1}}} \int_{F} \frac{\alpha|_{F}}{e_{S^{1}}(N_{F})}$,
\end{center}
where the sum is taken over all fixed components, and $e_{S^{1}}(N_{F})$ denotes the equivariant Euler class of the normal bundle to $F$ in $M$.
\end{theo}

For a compact oriented manifold $M$, the Atiyah–Singer index theorem states that the analytic index of the signature operator on $M$ is equal to its topological index. As an application, for a compact oriented $S^1$-manifold with a discrete fixed point set, the theorem yields the following formula.

\begin{theo} \emph{[Atiyah-Singer index theorem]} \cite{AS} \label{t22} Let the circle group $S^1$ act on a $2n$-dimensional compact oriented manifold $M$ with a discrete fixed point set. Then the signature of $M$ is
\begin{center}
$\displaystyle \textrm{sign}(M) = \sum_{p \in M^{S^1}} \epsilon(p) \prod_{i=1}^{n} \frac{1+t^{w_{p,i}}}{1-t^{w_{p,i}}}$
\end{center}
for all indeterminate $t$, and is a constant. \end{theo}

As an application of Theorem \ref{t22}, we obtain the following result.

\begin{pro}\label{p23}
Let the circle act on a compact oriented manifold $M$ with a discrete fixed point set. Suppose that $\dim M \equiv 2 \mod 4$. Then the signature of $M$ vanishes, and the number of fixed points $p$ with $\epsilon(p)=+1$ and the number of fixed point $p$ with $\epsilon(p)=-1$ are equal.
\end{pro}

\begin{proof}
By the Hirzebruch signature theorem \cite{H}, the signature of $M$ is equal to the $L$-genus of $M$. Since $\dim M \neq 0 \mod 4$, the $L$-genus of $M$ vanishes; thus $\mathrm{sign}(M)=0$. Taking $t=0$ in Theorem \ref{t22},
\begin{center}
$\displaystyle \textrm{sign}(M) = \sum_{p \in M^{S^1}} \epsilon(p)$
\end{center}
and this proposition follows. \end{proof}

Let $a$ be the smallest positive weight; $a=\min\{w_{p,i} \, | \, 1 \leq i \leq n, p \in M^{S^1}\}$. Manipulating the index formula of Theorem \ref{t22} as
\begin{center}
$\displaystyle \textrm{sign}(M) = \sum_{p \in M^{S^1}} \epsilon(p) \prod_{i=1}^{n} \frac{1+t^{w_{p,i}}}{1-t^{w_{p,i}}}= \sum_{p \in M^{S^1}} \epsilon(p) \prod_{i=1}^{n} [(1+t^{w_{p,i}}) \sum_{j=0}^n t^{jw_{p,i}}]=\sum_{p \in M^{S^1}} \epsilon(p) \prod_{i=1}^{n} (1+2 \sum_{j=0}^n t^{jw_{p,i}})$
\end{center}
and comparing the coefficients of the $t^a$-terms, the below lemma holds.

\begin{lem} \cite{J1, M} \label{l24}
Let the circle act on a compact oriented manifold $M$ with a discrete fixed point set. Let $a$ be the smallest positive weight. Then
\begin{center}
$\displaystyle \sum_{p \in M^{S^1}, \epsilon(p)=+1} N_p(a)=\sum_{p \in M^{S^1}, \epsilon(p)=-1} N_p(a)$,
\end{center}
where $N_p(a)=|\{i \, : \, w_{p,i}=a, 1 \leq i \leq n\}|$ is the number of times weight $a$ occurs at $p$.
\end{lem}

From now on, we discuss the orientability of isotropy submanifolds.
Consider an effective circle action on a compact oriented manifold $M$, and let $w \geq 2$ be an integer. As a subgroup of $S^1$, the group $\mathbb{Z}_w$ acts on $M$. The set $M^{\mathbb{Z}_w}$ of points in $M$ fixed by the $\mathbb{Z}_w$-action is a union of lower-dimensional closed submanifolds.
Suppose that the action on $M$ has isolated fixed points. If $\dim M = 4$, then any component of $M^{\mathbb{Z}_w}$ that contains an $S^1$-fixed point is orientable (see \cite[Lemma 2.1]{JM}).
On the other hand, for $2n > 4$, there exists an $S^1$-action on a $2n$-dimensional closed oriented manifold whose $\mathbb{Z}_2$-fixed component contains $S^1$-fixed points and is not orientable \cite{W}. This contrasts with the claim in \cite[Lemma 1]{HH} that for even $w$, any component of $M^{\mathbb{Z}_w}$ containing an $S^1$-fixed point is orientable.

\begin{lem} \label{codim-even}
Let the circle act effectively on an orientable manifold $M$. Let $w \geq 2$ be an integer. Then each component of $M^{\mathbb{Z}_w}$ has even codimension.
\end{lem}

\begin{proof}
Let $F$ be a component of $M^{\mathbb{Z}_w}$ and let $m$ be a point in $F$. The group $\mathbb{Z}_w$ acts trivially on the tangent space $T_mF$ of $F$ at $m$, and acts non-trivially on each irreducible of the normal space $N_mF$ of $F$ at $m$. The irreducible $\mathbb{Z}_w$-representations are $\mathbb{R}$, on which a generator of $\mathbb{Z}_w$ acts as multiplication by $-1$, or $\mathbb{C}$, on which a generator of $\mathbb{Z}_w$ acts as multiplication by a root of unity. Since the $S^1$-action on $T_mM$ is orientation preserving, the non-trivial irreducible representation of $Z_w$ of real dimension one appears with even multiplicity in $N_mF$. Therefore, the normal space $N_mF$ has even dimension.
\end{proof}

If $w$ is odd, then $M^{\mathbb{Z}_w}$ is orientable.

\begin{lem} \label{odd}
Let the circle act effectively on an orientable manifold $M$. Let $w \geq 3$ be an odd integer. Then each component of $M^{\mathbb{Z}_w}$ is orientable.
\end{lem}

\begin{proof}
Let $F$ be a component of $M^{\mathbb{Z}_w}$ and let $m$ be a point in $F$. Since $w \geq 3$ is odd, the normal space $N_mF$ of $F$ at $m$ decomposes into the sum of complex one-dimensional vector spaces, on which a generator of $\mathbb{Z}_w$ acts as a $w$-th root of unity. Thus, the $\mathbb{Z}_w$-action on $N_mF$ is orientation preserving. Since this holds for all points in $F$, the normal bundle $NF$ of $F$ in $M$ is orientable. Therefore, $F$ is orientable.
\end{proof}

When $w \geq 3$, each codimension 2 submanifold of $M^{\mathbb{Z}_w}$ is orientable.

\begin{lem} \label{codim2}
Let the circle group $S^1$ act on an orientable manifold $M$ with isolated fixed points. Let $w \geq 3$ be an integer. Let $F$ be a component of $M^{\mathbb{Z}_w}$. If $F$ has codimension 2, then $F$ is orientable.
\end{lem}

\begin{proof}
Since $F$ has codimension 2, for any $m \in F$, the normal space $N_mF$ of $F$ at $m$ in $M$ is a real 2-dimensional vector space. Since $w \geq 3$, the $\mathbb{Z}_w$-representation of $N_mF$ is isomorphic to a complex $1$-dimensional vector space on which the group $\mathbb{Z}_w$ acts by multiplication by a $w$-th root of unity. Hence this action on $N_mF$ is orientation preserving. Since this holds for all $m \in F$, it follows that the normal bundle $NF$ of $F$ in $M$ is orientable. Therefore, $F$ is orientable.
\end{proof}

Next, a $2$-dimensional component of $M^{\mathbb{Z}_w}$ containing an $S^1$-fixed point is also orientable.

\begin{lem} \label{dim2}
Let the circle group $S^1$ act on an orientable manifold $M$. Let $p$ be an isolated fixed point. Let $w \geq 2$ and let $F$ be a $2$-dimensional component of $M^{\mathbb{Z}_w}$ that contains $p$. Then $F$ is orientable.
\end{lem}

\begin{proof}
Since $M$ has an isolated fixed point $p$, the dimension of $M$ is even; this also follows from Lemma~\ref{codim-even} since $F$ is a component of $M^{\mathbb{Z}_w}$ and has dimension $2$.
To prove this lemma, assume on the contrary that $F$ is not orientable. The circle group $S^1$ acts on $F$ as a restriction, and this $S^1$-action on $F$ has a fixed point $p$ because the fixed point set of this $S^1$-action on $F$ is equal to the intersection of $F$ and $M^{S^1}$, that is, $F^{S^1}=F \cap M^{S^1}$. By the classification of $S^1$-actions on 2-dimensional compact manifolds (see, e.g., \cite{Au2}), $F$ is $\mathbb{RP}^2$; moreover, the action of $S^1/\mathbb{Z}_w$ on $F$ has $Z:=\mathbb{RP}^1$ as a component of the $\mathbb{Z}_2$-fixed point set, where we regard $\mathbb{Z}_2$ as a subgroup of $S^1/\mathbb{Z}_w$.
Then $Z$ is a $\mathbb{Z}_{2w}$-fixed component of $M$. This contradicts Lemma~\ref{codim-even}, since $Z$ has dimension one and $M$ has even dimension.
\end{proof}

It is well-known that if the number of fixed points is odd, the dimension of the manifold is a multiple of 4.

\begin{cor} \label{c27} \cite{J1} Let the circle act on a compact oriented manifold $M$. If the number of fixed points is odd, the dimension of $M$ is divisible by four. \end{cor}

If a circle action on a compact oriented manifold has two fixed points, the two fixed points have the same weights and have opposite signs; this easily follows from Theorem \ref{t22}.

\begin{theorem} \label{t28} \cite{Ko, J4}
Let the circle act on a compact oriented manifold with two fixed points $p$ and $q$. Then the weights at $p$ and $q$ are equal and $\epsilon(p)=-\epsilon(q)$. \end{theorem}

If the multiset of weights at every fixed point are the same, the number of fixed points must be even and the signature of the manifold vanishes.

\begin{theorem} \label{t29} \cite{J1}
Let the circle act on a $2n$-dimensional compact oriented manifold $M$ with a discrete fixed point set. Suppose that the weights at each fixed point are $\{a_1,\cdots,a_n\}$ for some positive integers $a_1,\cdots,a_n$. Then the number of fixed points $p$ with $\epsilon(p)=+1$ and that with $\epsilon(p)=-1$ are equal. Moreover, the signature of $M$ vanishes.
\end{theorem}

\begin{proof}
By Theorem \ref{t22},
\begin{center}
$\displaystyle \textrm{sign}(M) = \sum_{p \in M^{S^1}} \epsilon(p) \prod_{i=1}^{n} \frac{1+t^{w_{p,i}}}{1-t^{w_{p,i}}} = \sum_{p \in M^{S^1}} \epsilon(p) \prod_{i=1}^{n} \frac{1+t^{a_i}}{1-t^{a_i}}=\sum_{p \in M^{S^1}} \epsilon(p) \prod_{i=1}^{n} [(1+t^{a_i})\sum_{j=0}^{\infty} t^{ja_i}]=\sum_{p \in M^{S^1}} \epsilon(p) \prod_{i=1}^{n} (1+2\sum_{j=0}^{\infty} t^{ja_i})$
\end{center}
for all indeterminate $t$. Since the signature of $M$ is a constant, this theorem follows. \end{proof}

We describe an equivariant connected sum of two oriented $S^1$-manifolds, $M$ and $N$, at fixed points. Unlike the standard approach, we perform equivariant connected sums at several fixed points $p_1, \dots, p_k$ of $M$ and $q_1, \dots, q_k$ of $N$, where, for each $i$, the fixed points $p_i$ and $q_i$ have the same multiset of weights and opposite signs.

Let $M$ and $N$ be two $2n$-dimensional connected oriented $S^1$-manifolds with discrete fixed point sets. Suppose that for $i=1,\cdots,k$, $p_i \in M^{S^1}$ and $q_i \in N^{S^1}$ satisfy $\epsilon_{M}(p_i)=-\epsilon_{N}(q_i)$ and $\{w_{p,1},\cdots,w_{p,n}\}=\{w_{q,1},\cdots,w_{q,n}\}$. For each $i$, there is an equivariant diffeomorphism $f_i$ ($g_i$) from a unit disk $D_{2n}$ in $\mathbb{C}^n$ to a neighborhood of $p_i$ ($q_i$), where the circle acts on $\mathbb{C}^n$ by
\begin{center}
$g \cdot (z_1,\cdots,z_n)=(g^{w_{p,1}}z_1,\cdots,g^{w_{p,n}}z_n)$
\end{center}
for all $g \in S^1 \subset \mathbb{C}$ and for all $(z_1,\cdots,z_n) \in \mathbb{C}^n$.

\begin{Definition}
The \textbf{equivariant connected sum} of $M$ and $N$ (at $p_1,\cdots,p_k \in M$ and $q_1,\cdots,q_k \in N$) is the quotient 
\begin{center}
$\displaystyle \{(M \setminus \cup_{i=1}^k f_i(0)\} \sqcup \{N \setminus \cup_{i=1}^k g_i(0)\}/\sim$,
\end{center}
where we identify $f_i(tu)$ with $g_i((1-t)u)$ for each $u \in \partial D_{2n}$ and each $0 < t < 1$, for each $i$.
\end{Definition}

If we take the equivariant connected sum of $M$ and $N$ at $p_i$ of $M$ and $q_i$ of $M$ for all $i$, because $\epsilon_M(p_i)=-\epsilon_N(q_i)$ for $1 \leq i \leq k$ and each gluing map reverses orientation, we get an oriented $S^1$-manifold $P$ with fixed points $(M^{S^1} \setminus \{p_1,\cdots,p_k\}) \sqcup (N^{S^1} \setminus \{q_1,\cdots,q_k\})$, which is also connected. Consequently, the fixed point data of $P$ is $(\Sigma_M \setminus \cup_{i=1}^k \Sigma_{p_i}) \sqcup (\Sigma_N \setminus \cup_{i=1}^k \Sigma_{q_i})$.

\begin{lem} \label{l212}
Let $M$ and $N$ be two $2n$-dimensional connected oriented $S^1$-manifolds with discrete fixed point sets. Suppose that for $i=1,\cdots,k$, $p_i \in M^{S^1}$ and $q_i \in N^{S^1}$ satisfy $\epsilon_{M}(p_i)=-\epsilon_{N}(q_i)$ and $\{w_{p,1},\cdots,w_{p,n}\}=\{w_{q,1},\cdots,w_{q,n}\}$. The equivariant connected sum of $M$ and $N$ at $p_1$, $\cdots$, $p_k$ and $q_1$, $\cdots$, $q_k$ is a $2n$-dimensional connected oriented $S^1$-manifold $P$ with a discrete fixed point set, whose fixed point set is $(M^{S^1} \setminus \{p_1,\cdots,p_k\}) \sqcup (N^{S^1} \setminus \{q_1,\cdots,q_k\})$ and fixed point data is $(\Sigma_M \setminus \cup_{i=1}^k \Sigma_{p_i}) \sqcup (\Sigma_N \setminus \cup_{i=1}^k \Sigma_{q_i})$. If $M$ and $N$ are compact, so is $P$. \end{lem}

For a manifold $M$, let $\chi(M)$ denote the Euler number of $M$. Kobayashi proved that for a circle action on a compact manifold, its Euler number is equal to the sum of the Euler numbers of its fixed components.

\begin{theo} \cite{K} \label{t214}
Let the circle act on a compact oriented manifold $M$. Then \begin{center} $\displaystyle \chi(M)=\sum_{F \subset M^{S^1}} \chi(F)$, \end{center}
where the sum is taken over all fixed components of $M^{S^1}$.
\end{theo}

The Euler number of a compact oriented surface of genus $g$ is $2-2g$ and the Euler number of a point is 1. Therefore, Theorem \ref{t214} has the following consequence.

\begin{lem} \label{l215} \cite{J1}
Let $M$ be a compact connected oriented surface of genus $g$.
\begin{enumerate}
\item If $g=0$, i.e., $M$ is the 2-sphere $S^2$, then any non-trivial circle action on it has two fixed points.
\item If $g=1$, i.e., $M$ is the 2-torus $\mathbb{T}^2$, then any non-trivial circle action on it is fixed-point-free.
\item If $g>1$, then $M$ does not admit a non-trivial circle action.
\end{enumerate}
\end{lem}

\section{$S^6$, $\mathbb{CP}^3$, and 6-dimensional analogue $Z_n$ of the Hirzebruch surfaces} \label{s3}

In this section, we describe $S^1$-actions on $S^6$, $\mathbb{CP}^3$, 6-dimensional analogue $Z_n$ of the Hirzebruch surfaces, and $Z_2 \sharp \overline{Z_2}$ (a connected sum at fixed points of two copies of $Z_2$.)

\begin{exa}[The 6-sphere $S^6$]\label{e1}
Let $a$, $b$, and $c$ be positive integers. Let the circle act on $S^6$ by
\begin{center}
$g \cdot (z_1,z_2,z_3,x)=(g^a z_1, g^b z_2, g^c z_3,x)$
\end{center}
for all $g \in S^1 \subset \mathbb{C}$ and for all $(z_1,z_2,z_3,x) \in S^6$, where 
\begin{center}
$S^6=\{(z_1,z_2,z_3,x) \in \mathbb{C}^3 \times \mathbb{R} \, : \, x^2+\sum_{i=1}^3 |z_i|^2=1\}$.
\end{center}
The action has two fixed points $q_1=(0,0,0,1)$ and $q_2=(0,0,0,-1)$. The weights at $q_i$ are $\{a,b,c\}$, and $\epsilon(q_1)=-\epsilon(q_2)=1$. The fixed point data of this action on $S^6$ is hence
\begin{center}
$\{+,a,b,c\}$, $\{-,a,b,c\}$.
\end{center}
\end{exa}

\begin{exa}[The complex projective space $\mathbb{CP}^3$]\label{e2}
Let $0<a<b<c$ be positive integers. Let the circle act on $\mathbb{CP}^3$ by
\begin{center}
$g \cdot [z_0:z_1:z_2:z_3]=[z_0:g^a z_1:g^b z_2:g^c z_3]$
\end{center}
for all $g \in S^1 \subset \mathbb{C}$ and for all $[z_0:z_1:z_2:z_3] \in \mathbb{CP}^3$. 
The action has 4 fixed points $q_1=[1:0:0:0]$, $q_2=[0:1:0:0]$, $q_3=[0:0:1:0]$, and $q_4=[0:0:0:1]$, that have weights $\{a,b,c\}$, $\{-a,b-a,c-a\}$, $\{-b,a-b,c-b\}$, and $\{-c,a-c,b-c\}$ as complex $S^1$-representations, respectively.
The fixed point data of $\mathbb{CP}^3$ is
\begin{center}
$\{+,a,b,c\}$, $\{-,a,b-a,c-a\}$, $\{+,b,b-a,c-b\}$, $\{-,c,c-a,c-b\}$.
\end{center}
\end{exa}

\begin{exa}[The manifold $Z_n$, 6-dimensional analogue of the Hirzebruch surfaces]\label{e3}

Fix an integer $n$. By the 6-dimensional analogue $Z_n$ of Hirzebruch surfaces we mean a compact complex manifold
\begin{center}
$Z_n=\{([z_0:z_1:z_2:z_3],[w_2:w_3]) \in \mathbb{CP}^3 \times \mathbb{CP}^1 \, : \, z_2 w_3^n=z_3 w_2^n\}$.
\end{center}
Let $a$, $b$, and $c$ be positive integers such that $b-a \neq 0$, $nc-a \neq 0$, and $nc-b \neq 0$. Let the circle act on $Z_n$ by
\begin{center}
$g \cdot ([z_0:z_1:z_2:z_3],[w_2:w_3])=([g^a z_0: g^b z_1: z_2: g^{nc} z_3],[w_2:g^c w_3])$
\end{center}
for all $g \in S^1 \subset \mathbb{C}$ and for all $([z_0:z_1:z_2:z_3],[w_2:w_3]) \in Z_n$. We denote by $Z_n(a,b,c)$ the manifold with this action.
The action has 6 fixed points; at each fixed point, we exhibit local coordinates and the weights at the fixed point as complex $S^1$-representations.
\begin{enumerate}
\item $q_1=([1:0:0:0],[1:0])$: local coordinates $(z_1/z_0, z_2/z_0, w_3/w_2)$, weights $\{b-a, -a, c\}$
\item $q_2=([1:0:0:0],[0:1])$: local coordinates $(z_1/z_0, z_3/z_0, w_2/w_3)$, weights $\{b-a, nc-a, -c\}$
\item $q_3=([0:1:0:0],[1:0])$: local coordinates $(z_0/z_1, z_2/z_1, w_3/w_2)$, weights $\{a-b, -b, c\}$
\item $q_4=([0:1:0:0],[0:1])$: local coordinates $(z_0/z_1, z_3/z_1, w_2/w_3)$, weights $\{a-b, nc-b, -c\}$
\item $q_5=([0:0:1:0],[1:0])$: local coordinates $(z_0/z_2, z_1/z_2, w_3/w_2)$, weights $\{a, b, c\}$
\item $q_6=([0:0:0:1],[0:1])$: local coordinates $(z_0/z_3, z_1/z_3, w_2/w_3)$, weights $\{a-nc, b-nc, -c\}$
\end{enumerate}
For instance, local coordinates of $q_2$ are $(z_1/z_0, z_3/z_0, w_2/w_3)$, and the circle acts near $q_2$ by
\begin{center}
$\displaystyle g \cdot \left( \frac{z_1}{z_0}, \frac{z_3}{z_0}, \frac{w_2}{w_3}\right)=\left(  \frac{g^b z_1}{g^a z_0}, \frac{g^{nc} z_3}{g^a z_0}, \frac{w_2}{g^c w_3}\right)=\left(g^{b-a} \frac{z_1}{z_0}, g^{nc-a} \frac{z_3}{z_0}, g^{-c} \frac{w_2}{w_3}\right)$.
\end{center}
Thus the complex $S^1$-weights at $q_2$ are $\{b-a,nc-a,-c\}$.
\end{exa}

To simplify the proof of Theorems \ref{t11} and \ref{t12}, in Example \ref{e3} we will take specific values of $a$, $b$, $c$, and $n$ and record it as a separate example.

\begin{exa}[The manifold $Z_1(a,b,c)$ with $a>b>c>0$]\label{e4}
Suppose that $a>b>c>0$. Take the manifold $Z_1(a,b,c)$ in Example \ref{e3}. The weights at the fixed points as complex $S^1$-representations are
\begin{center}
$\{b-a, -a, c\}$, $\{b-a, c-a, -c\}$, $\{a-b, -b, c\}$, $\{a-b, c-b, -c\}$, $\{a, b, c\}$, $\{a-c, b-c, -c\}$,
\end{center} 
respectively. Therefore, as real $S^1$-representations, the fixed point data of $Z_1(a,b,c)$ is
\begin{center}
$\{+, a-b, a, c\}$, $\{-, a-b, a-c, c\}$, $\{-, a-b, b, c\}$, $\{+, a-b, b-c, c\}$, $\{+, a, b, c\}$, $\{-, a-c, b-c, c\}$.
\end{center}
\end{exa}

\begin{exa}[The manifold $Z_1(a,b,c)$ with $a>c>b>0$]\label{e5}
Suppose that $a>c>b>0$. Take the manifold $Z_1(a,b,c)$ in Example \ref{e3}. The weights at the fixed points as complex $S^1$-representations are
\begin{center}
$\{b-a, -a, c\}$, $\{b-a, c-a, -c\}$, $\{a-b, -b, c\}$, $\{a-b, c-b, -c\}$, $\{a, b, c\}$, $\{a-c, b-c, -c\}$.
\end{center} 
As real $S^1$-representations, the fixed point data of $Z_1(a,b,c)$ is
\begin{center}
$\{+, a-b, a, c\}$, $\{-, a-b, a-c, c\}$, $\{-, a-b, b, c\}$, $\{-, a-b, c-b, c\}$, $\{+, a, b, c\}$, $\{+, a-c, c-b, c\}$.
\end{center}
\end{exa}

The only difference between Examples \ref{e4} and \ref{e5} is some sign issue on the fixed point data of $q_4$ and $q_6$.

\begin{exa}[The manifold $Z_2(a,d,d)$ with $0<2d<a$]\label{e6}
Let $a$ and $d$ be positive integers such that $2d<a$. In Example \ref{e3}, take $n=2$ and $b=c=d$. Complex $S^1$-weights at the fixed points are
\begin{center}
$\{d-a, -a, d\}$, $\{d-a, 2d-a, -d\}$, $\{a-d, -d, d\}$, $\{a-d, d, -d\}$, $\{a, d, d\}$, $\{a-2d, -d, -d\}$.
\end{center} 
As real $S^1$-representations, the fixed point data of $Z_2(a,d,d)$ is
\begin{center}
$\{+, a-d, a, d\}$, $\{-,a-d,a-2d,d\}$, $\{-,a-d, d, d\}$, $\{-,a-d, d, d\}$, $\{+, a, d, d\}$, $\{+,a-2d, d, d\}$.
\end{center}
\end{exa}

\begin{exa}[The manifold $Z_2(a,d,d)$ with $2d>a>0$]\label{e7}
Let $a$ and $d$ be positive integers such that $2d>a$. In Example \ref{e3}, take $n=2$ and $b=c=d$. Complex $S^1$-weights at the fixed points are
\begin{center}
$\{d-a, -a, d\}$, $\{d-a, 2d-a, -d\}$, $\{a-d, -d, d\}$, $\{a-d, d, -d\}$, $\{a, d, d\}$, $\{a-2d, -d, -d\}$.
\end{center} 
As real $S^1$-representations, the fixed point data of $Z_2(a,d,d)$ is
\begin{center}
$\{+, a-d, a, d\}$, $\{+,a-d,2d-a,d\}$, $\{-,a-d, d, d\}$, $\{-,a-d, d, d\}$, $\{+, a, d, d\}$, $\{-,2d-a, d, d\}$.
\end{center}
\end{exa}

As for Examples \ref{e4} and \ref{e5}, the only difference between Examples \ref{e6} and \ref{e7} is some sign issue on the fixed point data of $q_2$ and $q_6$.

\begin{exa}[The manifold $Z_2(a,e,e) \sharp \overline{Z_2}(a,a-e,a-e)$]\label{e8}
Let $a$ and $e$ be positive integers such that $2e<a$. Take $Z_2(a,e,e)$ of Example \ref{e6} that has fixed point data
\begin{center}
$\{+, a-e, a, e\}$, $\{-,a-e,a-2e,e\}$, $\{-,a-e, e, e\}$, $\{-,a-e, e, e\}$, $\{+, a, e, e\}$, $\{+,a-2e, e, e\}$.
\end{center}
Denote the fixed points by $q_1',\cdots,q_6'$, respectively.
Since $2e<a$, it follows $a<2(a-e)$. We take $Z_2(a,a-e,a-e)$ (take $d=a-e$) of Example \ref{e7} that has fixed point data
\begin{center}
$\{+, e, a, a-e\}$, $\{+, e, a-2e, a-e\}$, $\{-, e, a-e, a-e\}$, $\{-, e, a-e, a-e\}$, $\{+, a, a-e, a-e\}$, $\{-, a-2e, a-e, a-e\}$.
\end{center}
We reverse the orientation of $Z_2(a,a-e,a-e)$ to get a manifold $\overline{Z_2}(a,a-e,a-e)$ that has fixed point data
\begin{center}
$\{-, e, a, a-e\}$, $\{-, e, a-2e, a-e\}$, $\{+,e,a-e,a-e\}$, $\{+,e,a-e,a-e\}$, $\{-, a, a-e, a-e\}$, $\{+, a-2e, a-e, a-e\}$.
\end{center}
Denote the fixed points by $q_1'',\cdots,q_6''$, respectively.
Now, $q_1'$ and $q_1''$ have the same weights (as real $S^1$-representations) and satisfy $\epsilon(q_1')=-\epsilon(q_1'')$. Therefore, we can take an equivariant connected sum at $q_1'$ of $Z_2(a,e,e)$ and at $q_1''$ of $\overline{Z_2}(a,a-e,a-e)$ to construct another 6-dimensional compact connected oriented $S^1$-manifold $Z_2(a,e,e) \sharp \overline{Z_2}(a,a-e,a-e)$ with 10 fixed points $\hat{q}_1$, $\cdots$, $\hat{q}_{10}$, that has fixed point data
\begin{center}
$\{-,a-e,a-2e,e\}$, $\{-,a-e, e, e\}$, $\{-,a-e, e, e\}$, $\{+, a, e, e\}$, $\{+,a-2e, e, e\}$, $\{-, e, a-2e, a-e\}$, $\{+,e,a-e,a-e\}$, $\{+,e,a-e,a-e\}$, $\{-, a, a-e, a-e\}$, $\{+, a-2e, a-e, a-e\}$.
\end{center}
\end{exa}

\section{Relation between fixed point datum of fixed points in isotropy submanifold} \label{s4}

To prove Theorems \ref{t11} and \ref{t12} for a 6-dimensional oriented $S^1$-manifold with isolated fixed points, we require some technical lemmas concerning the relationships between the fixed point data of fixed points lying in the same component of an isotropy submanifold $M^{\mathbb{Z}_l}$, where $l$ is the largest weight. To establish these lemmas, we need to introduce additional terminology.

Let the circle act effectively on a $2n$-dimensional compact oriented manifold $M$ with a discrete fixed point set. Let $w$ be a positive integer. Let $F$ be a component of $M^{\mathbb{Z}_w}$ such that $F \cap M^{S^1} \neq \emptyset$. 
Assume that $F$ is orientable, and choose orientations of $F$ and its normal bundle $NF$ so that the induced orientation on $TF \oplus NF$ agrees with the orientation of $M$. Let $p \in F \cap M^{S^1}$ be an $S^1$-fixed point. By permuting the $L_{p,i}$'s, we may write
\begin{center}
$T_pM=L_{p,1} \oplus \cdots \oplus L_{p,m} \oplus L_{p,m+1} \oplus \cdots \oplus L_{p,n}$,
\end{center}
where $T_pF=L_{p,1} \oplus \cdots \oplus L_{p,m}$ and $N_pF=L_{p,m+1} \oplus \cdots \oplus L_{p,n}$
The circle acts on each $L_{p,i}$ with weight $w_{p,i}$. As before, we orient each $L_{p,i}$ so that $w_{p,i}$ is positive.
\begin{Definition} \label{d41}
\begin{enumerate}[(1)]
\item $\epsilon_F(p) = +1$ if the orientation on $F$ agrees with the orientation on $L_{p,1} \oplus \cdots \oplus L_{p,m}$, and $\epsilon_F(p) = -1$ otherwise.
\item $\epsilon_N(p) = +1$ if the orientation on $NF$ agrees with the orientation on $L_{p,m+1} \oplus \cdots \oplus L_{p,n}$, and $\epsilon_N(p) = -1$ otherwise.
\end{enumerate}
\end{Definition}
By definition, $\epsilon(p) = \epsilon_F(p) \cdot \epsilon_N(p)$.

Now suppose that $\dim M = 6$ and that the largest weight $l$ is greater than 1. Assume that $M^{\mathbb{Z}_l}$ has a 2-dimensional component $F$ containing an $S^1$-fixed point $q$. Then $F$ is a 2-sphere and contains another fixed point $q'$. A relationship between the fixed point data of $q$ and $q'$ is as follows.

\begin{lemma} \label{l42}
Let the circle act effectively on a 6-dimensional compact oriented manifold $M$ with a discrete fixed point set. Suppose that the largest weight
$$l=\max\{w_{p,i} \, | \, 1 \leq i \leq n, p \in M^{S^1}\}$$
is greater than 1. Suppose there is a fixed point $q$ with weights $\{l, a, b\}$ for some positive integers $a$ and $b$ such that $a,b < l$. Then there exists another fixed point $q'$ such that one of the following holds:
\begin{enumerate}
\item $\epsilon(q') = -\epsilon(q)$ and the weights at $q'$ are $\{l, a, b\}$.
\item $\epsilon(q') = -\epsilon(q)$ and the weights at $q'$ are $\{l, l - a, l - b\}$.
\item $\epsilon(q') = \epsilon(q)$ and the weights at $q'$ are $\{l, a, l - b\}$.
\item $\epsilon(q') = \epsilon(q)$ and the weights at $q'$ are $\{l, l - a, b\}$.
\end{enumerate}
The fixed points $q$ and $q'$ lie in the same component of $M^{\mathbb{Z}_l}$, which is the 2-sphere.
\end{lemma}

\begin{proof}
Let $F$ be the component of $M^{\mathbb{Z}_l}$ that contains $q$, which is a lower-dimensional closed submanifold of $M$. Since $q$ has only one weight divisible b $l$, we have $\dim F=2$. By Lemma~\ref{dim2}, $F$ is orientable. Choose an orientation of $F$, and an orientation of $NF$, so that the induced orientation on $TF \oplus NF$ agrees with the orientation of $M$. The circle action on $M$ restricts to an $S^1$-action on $F$, under which $q$ is a fixed point. By Lemma~\ref{l215}, $F$ is a 2-sphere and has another fixed point, $q'$. Applying Theorem \ref{t28} to the induced action on $F$, we have $\epsilon_F(q)=-\epsilon_F(q')$. Since $NF$ is an oriented $\mathbb{Z}_w$-bundle over $F$, and $F$ is connected, the $\mathbb{Z}_w$-representations of $N_{q}F$ and $N_{q'}F$ are isomorphic. 

Let $N_qF=L_{q,2} \oplus L_{q,3}$, where the circle acts on $L_{q,2}$ with weight $a$ and on $L_{q,3}$ with weight $b$. Similarly, let $N_{q'}F=L_{q',2} \oplus L_{q',3}$, where the circle acts on $L_{q',2}$ with weight $c$ and on $L_{q',3}$ with weight $d$, for some positive integers $c$ and $d$. Note that $c,d<l$.

First, suppose that $\epsilon(q)=\epsilon(q')$. Since $\epsilon_F(q)=-\epsilon_F(q')$, with $\epsilon(q)=\epsilon_F(q) \cdot \epsilon_N(q)$ and $\epsilon(q')=\epsilon_F(q') \cdot \epsilon_N(q')$ this implies that $\epsilon_N(q)=-\epsilon_N(q')$. Therefore, there is an orientation reversing isomorphism $\phi$ from $L_{q,2} \oplus L_{q,3}$ to $L_{q',2} \oplus L_{q',3}$ as $\mathbb{Z}_l$-representations. Without loss of generality, by permuting $L_{q',2}$ and $L_{q',3}$ if necessary, we may assume that this isomorphism takes $L_{q,2}$ to $L_{q',2}$ and $L_{q,3}$ to $L_{q',3}$. Then one of the following holds.
\begin{enumerate}[(a)]
\item The isomorphism $\phi$ is orientation preserving from $L_{q,2}$ to $L_{q',2}$, and orientation reversing from $L_{q,3}$ to $L_{q',3}$.
\item The isomorphism $\phi$ is orientation reversing from $L_{q,2}$ to $L_{q',2}$, and orientation preserving from $L_{q,3}$ to $L_{q',3}$.
\end{enumerate}

Assume that Case (a) holds. Since $\phi$ is an orientation preserving isomorphism from $L_{q,2}$ to $L_{q',2}$ as $\mathbb{Z}_l$-representations, this implies that $a \equiv c \mod l$. Since $a,c<l$, it follows that $a=c$. Next, since $\phi$ is an orientation reserving isomorphism from $L_{q,3}$ to $L_{q',3}$ as $\mathbb{Z}_l$-representations, this implies that $b \equiv -d \mod l$. With $b,d<l$ this means that $d=l-b$. This is Case (3) of this lemma. Similarly, if Case (b) holds, then $c=l-a$ and $b=d$; this is Case (4) of this lemma.

Second, suppose that $\epsilon(q)=-\epsilon(q')$. Since $\epsilon_F(q)=-\epsilon_F(q')$, it follows that $\epsilon_N(q)=\epsilon_N(q')$. Hence there is an orientation preserving isomorphism $\phi$ from $L_{q,2} \oplus L_{q,3}$ to $L_{q',2} \oplus L_{q',3}$ as $\mathbb{Z}_l$-representations. We may assume that this isomorphism takes $L_{q,2}$ to $L_{q',2}$ and $L_{q,3}$ to $L_{q',3}$. There are two possibilities.
\begin{enumerate}[(i)]
\item The isomorphism $\phi$ is orientation preserving from $L_{q,i}$ to $L_{q',i}$ for both $i \in \{2,3\}$.
\item The isomorphism $\phi$ is orientation reversing from $L_{q,i}$ to $L_{q',i}$ for both $i \in \{2,3\}$.
\end{enumerate}
Case (i) means that $a=c$ and $b=d$; this is Case (1) of this lemma. Case (ii) means that $c=l-a$ and $d=l-b$; this is Case (2) of this lemma. 
\end{proof}

Suppose now that $M^{\mathbb{Z}_l}$ has a 4-dimensional component that contains an $S^1$-fixed point $q$. Equivalently, suppose that there is a fixed point $q$ that has weight $l$ twice. Then there must exist another fixed point $q'$ that has the same weights as $q$, and has the opposite sign; $\epsilon(q')=-\epsilon(q)$.

\begin{lemma} \label{l43}
Let the circle act effectively on a 6-dimensional compact oriented manifold $M$ with a discrete fixed point set. Suppose that the largest weight 
$$l=\max\{w_{p,i} \, | \, 1 \leq i \leq 3, p \in M^{S^1}\}$$
is greater than 2. 
\begin{enumerate}[(1)]
\item Suppose that there is a fixed point $q$ with weights $\{l,l,a\}$ for some positive integer $a$. Then there exists another fixed point $q'$ with weights $\{l,l,a\}$ such that $\epsilon(q')=-\epsilon(q)$. Moreover, $q$ and $q'$ are in the same component of $M^{\mathbb{Z}_l}$.
\item For each positive integer $a$, the number of fixed points with fixed point data $\{+,l,l,a\}$ is equal to the number of fixed points with fixed point data $\{-,l,l,a\}$.
\end{enumerate}
\end{lemma}

\begin{proof}
Suppose that a fixed point $p_1:=q$ has weights $\{l,l,a\}$ for some positive integer $a$. Since the action is effective and $l>2$ is the largest weight, $a$ is strictly smaller than $l$. Let $F$ be the component of $M^{\mathbb{Z}_l}$ that contains $p_1$. Since the multiplicity of the weight $l$ at $p_1$ is 2, we have $\dim F=4$. 
By Lemma \ref{codim2}, $F$ is orientable. Choose an orientation of $F$ so that $\epsilon_F(p_1)=+1$, for simplicity of the proof. Also choose an orientation of $NF$ so that the induced orientation on $TF \oplus NF$ agrees with the orientation of $M$. The circle action on $M$ restricts to a circle action on $F$, and the fixed point set $F^{S^1}$ of this action on $F$ is equal to $F \cap M^{S^1}$, so it is non-empty and finite. For any $p \in F^{S^1}$, the weights in $T_pF$ are multiples of $l$, and hence they are all equal to $l$, since $l$ is the largest weight. Applying Theorem \ref{t29} to the induced $S^1$-action on $F$, with $a_1=a_2=l$, the number of fixed points $p \in F^{S^1}$ with $\epsilon_F(p)=+1$ equals the number of fixed points $p \in F^{S^1}$ with $\epsilon_F(p)=-1$. Let $p_1,\cdots,p_k \in F^{S^1}$ be the fixed points with $\epsilon_F(p_i)=+1$, and let $q_1,\cdots,q_k \in F^{S^1}$ be the fixed points with $\epsilon_F(q_i)=-1$. 

By permuting $p_i$'s if necessary, let $\epsilon(p_1)=\cdots=\epsilon(p_s)$ and $\epsilon(p_{s+1})=\cdots=\epsilon(p_k)=-\epsilon(p_1)$. Similarly, by permuting $q_j$'s if necessary, let $\epsilon(q_1)=\cdots=\epsilon(q_t)=-\epsilon(p_1)$ and $\epsilon(q_{t+1})=\cdots=\epsilon(q_k)=\epsilon(p_1)$.

Let $\{l,l,a_i\}$ be the weights at $p_i$ and let $\{l,l,b_j\}$ be the weights at $q_j$. Then $a_i,b_j<l$. Since $NF$ is an oriented $\mathbb{Z}_w$-bundle over $F$ and $F$ is connected, the $\mathbb{Z}_w$-representations of $N_{p}F$ and $N_{p'}F$ are isomorphic for any two fixed points $p$ and $p'$ in $F^{S^1}$.

Consider $p_i$ for $1 \leq i \leq s$ ($s+1 \leq i \leq k$.) Since $\epsilon(p_i)=\epsilon(p_1)$ ($\epsilon(p_i)=-\epsilon(p_1)$) and $\epsilon_F(p_i)=\epsilon_F(p_1)$ ($\epsilon_F(p_i)=\epsilon_F(p_1)$), it follows that $\epsilon_N(p_i)=\epsilon_N(p_1)$ ($\epsilon_N(p_i)=-\epsilon_N(p_1)$.) Thus, there is an orientation preserving (reversing) isomorphism from the representation $N_{p_1}F$ with weight $a$ to that $N_{p_i}F$ with weight $a_i$ as $\mathbb{Z}_l$-representations. This implies that $a \equiv a_i \mod l$ ($a \equiv -a_i \mod l$) and hence $a=a_i$ ($a_i=l-a$, respectively.)

Similarly, for $q_j$ with $1 \leq j \leq t$ ($t+1 \leq j \leq k$), $\epsilon(q_j)=-\epsilon(p_1)$ ($\epsilon(q_j)=\epsilon(p_1)$) and $\epsilon_F(q_j)=-\epsilon_F(p_1)$ imply that $\epsilon_N(q_j)=\epsilon_N(p_1)$ ($\epsilon_N(q_j)=-\epsilon_N(p_1)$) and hence there is an orientation preserving (reversing) isomorphism from $N_{p_1}F$ with weight $a$ to $N_{q_j}F$ with weight $b_j$ as $\mathbb{Z}_l$-representations, and thus $a \equiv b_j \mod l$ ($a \equiv -b_j \mod l$), i.e., $b_j=a$ ($b_j=l-a$, respectively.)

Let $e_{S^1}(NF)$ denote the equivariant Euler class of the normal bundle to $F$ in $M$. Let $u$ be a degree two generator of $H^*(\mathbb{CP}^{\infty};\mathbb{Z})$. The restriction of $e_{S^1}(NF)$ at $p_i$ is 
\begin{enumerate}
\item $e_{S^1}(NF)(p_i)=\epsilon_N(p_i) \cdot au=\epsilon_N(p_1) \cdot au$ if $1 \leq i \leq s$.
\item $e_{S^1}(NF)(p_i)=\epsilon_N(p_i) \cdot (l-a)u=-\epsilon_N(p_1) \cdot (l-a)u$ if $s+1 \leq i \leq k$.
\end{enumerate}
Similarly, the restriction of $e_{S^1}(NF)$ at $q_j$ is
\begin{enumerate}
\item $e_{S^1}(NF)(q_j)=\epsilon_N(q_j) \cdot au=\epsilon_N(p_1) \cdot au$ if $1 \leq j \leq t$.
\item $e_{S^1}(NF)(q_j)=\epsilon_N(q_j) \cdot (l-a)u=-\epsilon_N(p_1) \cdot (l-a)u$ if $t+1 \leq j \leq k$.
\end{enumerate}

For the induced action on $F$, $\int_F:=\pi_*$ is a map from $H_{S^1}^i(F;\mathbb{Z})$ to $H^{i- \dim F}(\mathbb{CP}^\infty;\mathbb{Z})$ for all $i \in \mathbb{Z}$. Since $\dim F=4$ and $e_{S^1}(NF)$ has degree 2, the image under $\int_F$ of $e_{S^1}(NF)$ vanishes;
\begin{center}
$\displaystyle \int_F e_{S^1}(NF)=0$.
\end{center}
On the other hand, applying the ABBV localization formula (Theorem \ref{t21}) to the induced action on $F$ with taking $\alpha=e_{S^1}(NF)$,
\begin{center}
$\displaystyle \int_F e_{S^1}(NF)=\sum_{p \in F^{S^1}}\int_p \frac{e_{S^1}(NF)|_p}{e_{S^1}(T_pF)}$\footnote{Note that the bundle in the denominator is the normal bundle of $p$ in $F$, which is therefore the tangent space $T_pF$ of $p$ in $F$.}

$\displaystyle =\sum_{i=1}^k \bigg\{ \epsilon_F(p_i)\frac{e_{S^1}(NF)(p_i)}{e_{S^1}(T_{p_i}F)}\bigg\}+ \sum_{j=1}^k \bigg\{ \epsilon_F(q_j)\frac{e_{S^1}(NF)(q_j)}{e_{S^1}(T_{q_j}F)}\bigg\}$

$\displaystyle =\sum_{i=1}^s \bigg\{ \epsilon_F(p_i)\frac{e_{S^1}(NF)(p_i)}{e_{S^1}(T_{p_i}F)} \bigg\} + \sum_{i=s+1}^k \bigg\{ \epsilon_F(p_i)\frac{e_{S^1}(NF)(p_i)}{e_{S^1}(T_{p_i}F)} \bigg\}$

$\displaystyle + \sum_{j=1}^t \bigg\{ \epsilon_F(q_j)\frac{e_{S^1}(NF)(q_j)}{e_{S^1}(T_{q_j}F)}\bigg\} + \sum_{j=t+1}^k \bigg\{ \epsilon_F(q_j)\frac{e_{S^1}(NF)(q_j)}{e_{S^1}(T_{q_j}F)}\bigg\}$

$\displaystyle =\sum_{i=1}^s \bigg\{ +\frac{+\epsilon_N(p_1) \cdot au}{l^2u^2}\bigg\} +\sum_{i=s+1}^k \bigg\{ +\frac{-\epsilon_N(p_1) \cdot (l-a)u}{l^2u^2}\bigg\}$

$\displaystyle +\sum_{j=1}^t \bigg\{ -\frac{+\epsilon_N(p_1) \cdot au}{l^2u^2} \bigg\} + \sum_{j=t+1}^k \bigg\{ -\frac{-\epsilon_N(p_1) \cdot (l-a)u}{l^2u^2}\bigg\}$

$\displaystyle =\frac{\epsilon_N(p_1)}{l^2 u} \{sa - (k-s)(l-a) - ta +(k-t)(l-a)\}$

$\displaystyle =\frac{\epsilon_N(p_1)}{l u}(s-t).$
\end{center}
Therefore, $s=t$. Then $p_1,\cdots,p_s$ and $q_1,\cdots,q_s$ have weights $\{l,l,a\}$ and $\epsilon(p_1)=\cdots=\epsilon(p_s)=-\epsilon(q_1)=\cdots=-\epsilon(q_s)$. Also, $p_{s+1},\cdots,p_k$ and $q_{s+1},\cdots,q_k$ have weights $\{l,l,l-a\}$ and $\epsilon(p_{s+1})=\cdots=\epsilon(p_k)=-\epsilon(q_{s+1})=\cdots=-\epsilon(q_k)=-\epsilon(p_1)$. Thus, this lemma holds. \end{proof}

\section{Proof of main results} \label{s5}

In this section, we prove Theorem \ref{t11} and Theorem \ref{t12} together.

\begin{proof}[\textbf{Proof of Theorems \ref{t11} and \ref{t12}}]

We proceed by showing that fixed points of $M$ with the largest weight can be successively removed via connected sums at fixed points of $S^1$-actions on $S^6$, $\mathbb{CP}^3$, $Z_1$, and $Z_2$ (as well as these manifolds with opposite orientations).

By quotienting out by the subgroup that acts trivially, we may assume that the action is effective. Let $l$ be the largest weight, that is,
$$l=\max\{w_{p,i} \, | \, 1 \leq i \leq 3, p \in M^{S^1}\}.$$
If $l=1$ or $l=2$, then both Theorems \ref{t11} and \ref{t12} follow from Lemma \ref{l51}, where we use $S^6$ for the equivariant connected sum. In terms of the operations in Theorem \ref{t12}, this corresponds to Operation (1). 

Therefore, from now on, we assume that $l>2$. Suppose that a fixed point $p_1$ has weight $l$. Without loss of generality, by reversing the orientation of $M$ if necessary, we may assume that $\epsilon(p_1)=+1$. This assumption is made to simplify the proof, but a caution is in order: in the proof below, we will take the connected sum of $M$ with some $N \in \{S^6, \mathbb{CP}^3, \overline{\mathbb{CP}^3}, Z_1, \overline{Z_1}, Z_2, \overline{Z_2}, Z_2 \sharp \overline{Z_2}\}$. If instead $\epsilon(p_1)=-1$, then we must take the connected sum of $M$ and $\overline{N}$. 

Let $F$ be the component of $M^{\mathbb{Z}_l}$ that contains $p_1$.
Since the action is effective, there are two possibilities.
\begin{enumerate}[(1)]
\item The multiplicity of $l$ in $T_{p_1}M$ is 2. 
\item The multiplicity of $l$ in $T_{p_1}M$ is 1. 
\end{enumerate}

Assume Case (1) holds. Let $\{l,l,x\}$ be the weights at $p_1$, for some positive integer $x$. Then $x<l$. By Lemma \ref{l43}, there exists another fixed point $p_2$ such that $\epsilon(p_2)=-\epsilon(p_1)=-1$, the weights at $p_2$ are also $\{l,l,x\}$. Then we can perform an equivariant connected sum at $p_1$ and $p_2$ of $M$ and $q_2$ and $q_1$ of $S^6$ in Example \ref{e1} with $\{a,b,c\}=\{l,l,x\}$, to construct another 6-dimensional compact connected oriented $S^1$-manifold $M'$ that has fixed point data $$\Sigma_M \setminus (\{+,l,l,x\} \cup \{-,l,l,x\}).$$ This corresponds to Operation (1) of Theorem \ref{t12}, where $\{l,l,x\}$ corresponds to $\{A,B,C\}$ in Operation (1).

Assume Case (2) holds. Let $\{l,x,y\}$ be the weights at $p_1$ for some positive integers $x$ and $y$. Then $x,y<l$. By Lemma \ref{l42}, there exists another fixed point $p_2$ such that one of the following holds:
\begin{enumerate}[(a)]
\item $\epsilon(p_2)=-\epsilon(p_1)=-1$ and the weights at $p_2$ are $\{l,x,y\}$.
\item $\epsilon(p_2)=-\epsilon(p_1)=-1$ and the weights at $p_2$ are $\{l,l-x,l-y\}$.
\item $\epsilon(p_2)=\epsilon(p_1)=1$ and the weights at $p_2$ are $\{l,x,l-y\}$.
\item $\epsilon(p_2)=\epsilon(p_1)=1$ and the weights at $p_2$ are $\{l,l-x,y\}$.
\end{enumerate}
Up to permuting $x$ and $y$, Case (c) and Case (d) are equivalent; thus, we only need to consider Cases (a-c).

Assume that Case (2-a) holds. The fixed points $p_1$ and $p_2$ have fixed point data $\{+,l,x,y\}$ and $\{-,l,x,y\}$, respectively. As in Case (1), we can take an equivariant connected sum at $p_1$ and $p_2$ of $M$ and $q_2$ and $q_1$ of $S^6$ in Example \ref{e1} with $\{a,b,c\}=\{l,x,y\}$, to construct another 6-dimensional compact connected oriented $S^1$-manifold $M'$ whose fixed point data is 
$$\Sigma_M \setminus (\{+,l,x,y\} \cup \{-,l,x,y\}).$$
This corresponds to Operation (1) of Theorem \ref{t12}.

Assume that Case (2-b) holds. We have two subcases.
\begin{enumerate}[(i)]
\item $x \neq y$.
\item $x=y$.
\end{enumerate}

Suppose that Case (2-b-i) holds. Permuting $x$ and $y$ if necessary, we may assume that $x>y$. The fixed points $p_1$ and $p_2$ have fixed point data $\{+,l,x,y\}$ and $\{-,l,l-x,l-y\}$, respectively. In Example \ref{e2} of the action on $\mathbb{CP}^3$, take $c=l$, $b=x$, and $a=y$, and reverse its orientation; its fixed point data is
\begin{center}
$\Sigma_{q_1}=\{-,y,x,l\}$, $\Sigma_{q_2}=\{+,y,x-y,l-y\}$, $\Sigma_{q_3}=\{-,x,x-y,l-x\}$, $\Sigma_{q_4}=\{+,l,l-y,l-x\}$.
\end{center}
Then we can take an equivariant connected sum at $p_1$ and $p_2$ of $M$ and $q_1$ and $q_4$ of $\overline{\mathbb{CP}^3}$ (with this action), to construct another oriented $S^1$-manifold $M'$ whose fixed point data is
\begin{center}
$\{\Sigma_M \setminus (\{+,l,x,y\} \cup \{-,l,l-y,l-x\})\} \cup (\{+,y,x-y,l-y\} \cup \{-,x,x-y,l-x\})$. 
\end{center}
This corresponds to Operation (2) of Theorem \ref{t12} with $l=C$, $x=B$, and $y=A$.

Suppose that Case (2-b-ii) holds. The fixed points $p_1$ and $p_2$ have fixed point data $\{+,l,x,x\}$ and $\{-,l,l-x,l-x\}$, respectively. First, suppose that $2x<l$. In Example \ref{e8} of the action on $Z:=Z_2(a,e,e) \sharp \overline{Z_2}(a,a-e,a-e)$ we take $a=l$ and $e=x$ and reverse its orientation; its fixed point data is
\begin{center}
$\{+,l-x,l-2x,x\}$, $\{+,l-x, x, x\}$, $\{+,l-x, x, x\}$, $\{-, l, x, x\}$, $\{-,l-2x, x, x\}$, $\{+, x, l-2x, l-x\}$, $\{-,x,l-x,l-x\}$, $\{-,x,l-x,l-x\}$, $\{+, l, l-x, l-x\}$, $\{-, l-2x, l-x, l-x\}$.
\end{center}
We can take a connected sum at $p_1$ and $p_2$ of $M$ and $\hat{q}_4$ and $\hat{q}_9$ of $\overline{Z}$, to construct another $S^1$-manifold $M'$ whose fixed point data is
\begin{center}
$\Sigma_{M'}=\{\Sigma_M \setminus (\{+,l,x,x\} \cup \{-,l,l-x,l-x\})\} \cup (\{+,l-x,l-2x,x\} \cup \{+,l-x, x, x\} \cup \{+,l-x, x, x\} \cup \{-,l-2x, x, x\} \cup \{+, x, l-2x, l-x\} \cup \{-,x,l-x,l-x\} \cup \{-,x,l-x,l-x\} \cup \{-, l-2x, l-x, l-x\})$. 
\end{center}
This corresponds to Operation (5) of Theorem \ref{t12} with $l=C$ and $x=A$.

Second, assume Case (2-b-ii) and $2x>l$. In Example \ref{e8} of the action on $Z_2(a,e,e) \sharp \overline{Z_2}(a,a-e,a-e)$ we take $a=l$ and $e=l-x$; its fixed point data is
\begin{center}
$\{-,x,2x-l,l-x\}$, $\{-,x, l-x, l-x\}$, $\{-, x, l-x, l-x\}$, $\{+, l, l-x, l-x\}$, $\{+, 2x-l, l-x, l-x\}$, $\{-, l-x, 2x-l, x\}$, $\{+, l-x, x, x\}$, $\{+, l-x, x, x\}$, $\{-, l, x, x\}$, $\{+, 2x-l, x, x\}$.
\end{center}
We can take a connected sum at $p_1$ and $p_2$ of $M$ and $\hat{q}_9$ and $\hat{q}_4$ of $Z_2(l, l-x, l-x) \sharp \overline{Z_2}(l, x, x)$, to construct another $S^1$-manifold $M'$ with fixed point data 
\begin{center}
$\Sigma_{M'}=\{\Sigma_M \setminus (\{+,l,x,x\} \cup \{-,l,l-x,l-x\})\} \cup (\{-,x,2x-l,l-x\} \cup \{-,x, l-x, l-x\} \cup \{-, x, l-x, l-x\} \cup \{+, 2x-l, l-x, l-x\} \cup \{-, l-x, 2x-l, x\} \cup \{+, l-x, x, x\} \cup \{+, l-x, x, x\} \cup \{+, 2x-l, x, x\})$. 
\end{center}
This also corresponds to Operation (5) of Theorem \ref{t12} with $l=C$ and $x=C-A$.

Third, assume Case (2-b-ii) and $2x=l$. Since $p_1$ has weights $\{2x(=l),x,x\}$ and the action is effective, this implies that $x=1$. Then $p_1$ and $p_2$ have fixed point data $\{+,2,1,1\}$ and $\{-,2,1,1\}$ and this case is Case (2-a); thus proceed as in Case (2-a).

Assume that Case (2-c) holds. We have two subcases.
\begin{enumerate}[(i)]
\item $x \neq y$.
\item $x=y$.
\end{enumerate}

Suppose that Case (2-c-i) holds. The fixed points $p_1$ and $p_2$ have fixed point data $\{+,l,x,y\}$ and $\{+,l,x,l-y\}$, respectively. First, suppose that $x<y$. In Example \ref{e4} of the action on $Z_1(a,b,c)$ we take $a=l$, $b=y$, and $c=x$, and reverse its orientation; its fixed point data is
\begin{center}
$\{-, l-y, l, x\}$, $\{+, l-y, l-x, x\}$, $\{+, l-y, y, x\}$, $\{-, l-y, y-x, x\}$, $\{-, l, y, x\}$, $\{+, l-x, y-x, x\}$.
\end{center}
We can take a connected sum at $p_1$ and $p_2$ of $M$ and $q_5$ and $q_1$ of $\overline{Z_1(l,y,x)}$ to construct another $S^1$-manifold $M'$ with fixed point data
\begin{center}
$\Sigma_{M'}=\{\Sigma_M \setminus (\{+,l,x,y\} \cup \{+,l,x,l-y\})\} \cup (\{+, l-y, l-x, x\} \cup \{+, l-y, y, x\} \cup \{-, l-y, y-x, x\} \cup \{+, l-x, y-x, x\})$.
\end{center}
This corresponds to Operation (3) of Theorem \ref{t12} with $x=A$, $y=B$, and $l=C$.

Next, suppose Case (2-c-i) with $x>y$. In Example \ref{e5} of the action on $Z_1(a,b,c)$ we take $a=l$, $b=y$, and $c=x$ and reverse its orientation; its fixed point data is
\begin{center}
$\{-, l-y, l, x\}$, $\{+, l-y, l-x, x\}$, $\{+, l-y, y, x\}$, $\{+, l-y, x-y, x\}$, $\{-, l, y, x\}$, $\{-, l-x, x-y, x\}$.
\end{center}
We can take a connected sum at $p_1$ and $p_2$ of $M$ and $q_5$ and $q_1$ of $\overline{Z_1(l,y,x)}$ to construct another $S^1$-manifold $M'$ with fixed point data
\begin{center}
$\Sigma_{M'}=\{\Sigma_M \setminus (\{+,l,x,y\} \cup \{+,l,x,l-y\})\} \cup (\{+, l-y, l-x, x\} \cup \{+, l-y, y, x\} \cup \{+, l-y, x-y, x\} \cup \{-, l-x, x-y, x\})$.
\end{center}
This corresponds to Operation (3') of Theorem \ref{t12} with $x=A$, $y=B$, and $l=C$.

Suppose that Case (2-c-ii) holds. The fixed points $p_1$ and $p_2$ have fixed point data $\{+,l,x,x\}$ and $\{+,l,x,l-x\}$, respectively. First, suppose that $2x<l$. In Example \ref{e6} we take $a=l$ and $d=x$ and reverse its orientation; its fixed point data is
\begin{center}
$\{-, l-x, l, x\}$, $\{+,l-x,l-2x,x\}$, $\{+,l-x, x, x\}$, $\{+,l-x, x, x\}$, $\{-, l, x, x\}$, $\{-,l-2x, x, x\}$.
\end{center}
We can take a connected sum at $p_1$ and $p_2$ of $M$ and $q_5$ and $q_1$ of $\overline{Z_2(l,x,x)}$ to construct another $S^1$-manifold $M'$ with fixed point data
\begin{center}
$\Sigma_{M'}=\{\Sigma_M \setminus (\{+,l,x,x\} \cup \{+,l,x,l-x\})\} \cup (\{+,l-x,l-2x,x\} \cup \{+,l-x, x, x\} \cup \{+,l-x, x, x\} \cup \{-,l-2x, x, x\} )$.
\end{center}
This corresponds to Operation (4) of Theorem \ref{t12} with $x=A$ and $l=C$. 

Second, suppose Case (2-c-ii) with $2x>l$. In Example \ref{e7} we take $a=l$ and $d=x$ and reverse its orientation; its fixed point data is
\begin{center}
$\{-, l-x, l, x\}$, $\{-,l-x,2x-l,x\}$, $\{+,l-x, x, x\}$, $\{+,l-x, x, x\}$, $\{-, l, x, x\}$, $\{+,2x-l, x, x\}$.
\end{center}
We can take a connected sum at $p_1$ and $p_2$ of $M$ and $q_5$ and $q_1$ of $\overline{Z_2(l,x,x)}$ to construct another $S^1$-manifold $M'$ with fixed point data
\begin{center}
$\Sigma_{M'}=\{\Sigma_M \setminus (\{+,l,x,x\} \cup \{+,l,x,l-x\})\} \cup (\{-,l-x,2x-l,x\} \cup \{+,l-x, x, x\} \cup \{+,l-x, x, x\} \cup \{+,2x-l, x, x\} )$.
\end{center}
This corresponds to Operation (4') of Theorem \ref{t12} with $x=A$ and $l=C$.

Third, suppose Case (2-c-ii) with $2x=l$. Since $p_1$ has weights $\{2x,x,x\}$ and the action is effective, this implies that $x=1$, that is, the largest weight $l$ is 2. By Lemma \ref{l51} below, we can take an equivariant connected sum at fixed points of $M$ and fixed points of rotations of $S^6$'s to construct a fixed-point-free $S^1$-action on a compact connected oriented manifold $M'$, which corresponds to taking Operation (1) of Theorem \ref{t12} repeatedly.

To sum up, by repeatedly applying the above steps, we can successively take equivariant connected sums at two fixed points of $M$ and at two fixed points of $S^1$-actions on $S^6$, $\mathbb{CP}^3$, $Z_1$, $Z_2$, and $Z_2 \sharp \overline{Z_2}$ (and these with opposite orientations) to construct another 6-dimensional compact connected oriented $S^1$-manifold $M''$ with a discrete fixed point set, in which the largest weight is strictly smaller than $l$. This is because every weight added at each step is smaller than $l$. Now, on $M''$, let $$l''=\max\{w_{p,i} \, | \, 1 \leq i \leq 3, p \in (M'')^{S^1}\}<l$$ and repeat the above argument.

In the end, by successively taking equivariant connected sums with the manifolds described above, we obtain a 6-dimensional compact connected oriented $S^1$-manifold $\widehat{M}$ with a discrete fixed point set, in which every weight at any fixed point is 1 or 2, provided the fixed point set is non-empty. As in the beginning of this proof, by Lemma \ref{l51} we can take an equivariant connected sum at fixed points of $\widehat{M}$ and at fixed points of rotations of $S^6$'s to construct a fixed-point-free $S^1$-action on another 6-dimensional compact connected oriented manifold. This last step corresponds to Operation (1) of Theorem \ref{t12}. \end{proof}

In the proof of Theorems \ref{t11} and \ref{t12} above, we used the following lemma.

\begin{lemma} \label{l51}
Let the circle act on a 6-dimensional compact connected oriented manifold $M$ with a discrete fixed point set. Suppose that the largest weight at most 2, that is, $$\max\{w_{p,i} \, | \, 1 \leq i \leq 3, p \in M^{S^1}\} \leq 2.$$ Then the number of fixed points with fixed point data $\{+,1,1,1\}$ (respectively, $\{+,1,1,2\}$ and $\{+,1,2,2\}$) is equal to the number of fixed points with fixed point data $\{-,1,1,1\}$ (respectively, $\{-,1,1,2\}$ and $\{-,1,2,2\}$). Consequently, we can take equivariant connected sums at fixed points of $M$ and at fixed points of rotations of $S^{6}$'s to construct a fixed-point-free $S^1$-action on a 6-dimensional compact connected oriented manifold.
\end{lemma}

\begin{proof}
Since the largest weight is at most 2, the possible multisets of weights at each fixed point are $\{1,1,1\}$, $\{1,1,2\}$, and $\{1,2,2\}$. Let $k_1,\cdots,k_6$ denote the numbers of fixed points with fixed point data $\{+,1,1,1\}$, $\{+,1,1,2\}$, $\{+,1,2,2\}$, $\{-,1,1,1\}$, $\{-,1,1,2\}$, and $\{-,1,2,2\}$, respectively. By Proposition \ref{p23}, the number $k_1+k_2+k_3$ of fixed points with sign $+1$ is equal to the number $k_4+k_5+k_6$ of fixed points with sign $-1$, that is, (1) $k_1+k_2+k_3=k_4+k_5+k_6$. Applying Lemma \ref{l24} to $M$ with $a=1$, it follows that (2) $3k_1+2k_2+k_3=3k_4+2k_5+k_6$.

Taking $\alpha=1$ in Theorem \ref{t21},
\begin{center}
$\displaystyle 1=\sum_{p \in M^{S^1}} \frac{1}{\prod_{i=1}^3 w_{p,i}}=k_1 + k_2 \frac{1}{2} + k_3 \frac{1}{4} - k_4 - k_5 \frac{1}{2} - k_6 \frac{1}{4}$.
\end{center}
On the other hand, since the equivariant cohomology class 1 has degree 0, by a dimensional reason that $\int_M$ is a map from $H_{S^1}^i(M;\mathbb{Z})$ to $H^{i- \dim M} (\mathbb{CP}^\infty;\mathbb{Z})$, the image of 1 under $\int_M$ vanishes, that is, $\int_M 1=0$.
Thus, $0=k_1 + k_2 \frac{1}{2} + k_3 \frac{1}{4} - k_4 - k_5 \frac{1}{2} - k_6 \frac{1}{4}$, that is, (3) $4k_1+2k_2+k_3=4k_4+2k_5+k_6$. Then (1-3) imply that $k_1=k_4$, $k_2=k_5$, and $k_3=k_6$.

Therefore, for each pair of fixed points $(p_i,q_i)$ that have the same multiset of weights $\{a,b,c\}$ and $\epsilon(p_i)=-\epsilon(q_i)=1$, we can take an equivariant connected sum at $p_i$ and $q_i$ of $M$ and at fixed points of a rotation of $S^6$ in Example \ref{e1}, to construct a fixed-point-free $S^1$-action on a 6-dimensional compact connected oriented manifold.
\end{proof}


\begin{thebibliography}{1}

\bibitem[Ah]{Ah}
K. Ahara: \emph{6-dimensional almost complex $S^1$-manifolds with $\chi(M)=4$.} J. Fac. Sci. Univ. Tokyo Sect. IA, Math. \textbf{38} (1991) no.1, 47--72.

\bibitem[Au1]{Au}
M. Audin: \emph{Hamiltoniens p\'eriodiques sur les vari\'et\'es symplectiques compactes de dimension 4}. G\'eom\'etrie symplectique et m\'ecanique, Proceedings 1988, C. Albert ed., Springer Lecture Notes in Math. 1416 (1990).

\bibitem[Au2]{Au2}
M. Audin: \emph{Torus Actions on Symplectic Manifolds}, Progress in Mathematics \textbf{93}, Birkh\"auser Verlag, Basel, 2004.

\bibitem[AB]{AB}
M. Atiyah and R. Bott: \emph{The moment map and equivariant cohomology.} Topology \textbf{23} (1984) 1--28.

\bibitem[AH]{AH}
K. Ahara and A. Hattori: \emph{4 dimensional symplectic $S^1$-manifolds admitting moment map.} J. Fac. Sci. Univ. Tokyo Sect. IA, Math. 38 (1991) 251--298.

\bibitem[AS]{AS}
M. Atiyah and I. Singer: \emph{The index of elliptic operators: III.} Ann. Math \textbf{87} (1968) 546--604.

\bibitem[BV]{BV}
N. Berline and M. Vergne: \emph{Classes caract{\'e}ristiques {\'e}quivariantes. {Formule} de localisation en cohomologie {\'e}quivariante.} C. R. Acad. Sci., Paris, S{\'e}r. I, \textbf{295} (1982), 539--541.

\bibitem[CHK]{CHK}
J. Carrell, A. Howard, and C. Kosniowski: \emph{Holomorphic vector fields on complex surfaces.} Math. Ann. \textbf{204} (1973) 73--81.

\bibitem[F1]{F1}
R. Fintushel: \emph{Locally smooth circle actions on homotopy 4-spheres}. Duke Math. J. \textbf{43} (1976) no.1, 63--70.

\bibitem[F2]{F2}
R. Fintushel: \emph{Circle actions on simply connected 4-manifolds}. Trans. Amer. Math. Soc. \textbf{230} (1977) 147--171.

\bibitem[F3]{F3}
R. Fintushel: \emph{Classification of circle actions on 4-manifolds}. Trans. Amer. Math. Soc. \textbf{232} (1978) 377--390.

\bibitem[H]{H}
F. Hirzebruch: \emph{Topological methods in algebraic geometry.} Grundlehren der Math. Wiss. \textbf{131}, Springer-Verlag 1966.

\bibitem[HH]{HH}
H. Herrera and R. Herrera: \emph{$\hat{A}$-genus on non-spin manifolds with $S^1$ actions and the classification of positive quaternion-K\"ahler 12-manifolds.} J. Differ. Geom. \textbf{61} (2002) 341--364.

\bibitem[J1]{J1}
D. Jang: \emph{Circle actions on oriented manifolds with discrete fixed point sets and classification in dimension 4.} J. Geom. Phys. \textbf{133} (2018) 181--194.

\bibitem[J2]{J2}
D. Jang: \emph{Circle actions on almost complex manifolds with 4 fixed points.} Math. Z. \textbf{294} (2020) 287--319.

\bibitem[J3]{J4}
D. Jang: \emph{Circle actions on oriented manifolds with few fixed points.} East Asian Math. J. \textbf{36} (2020), No. 5, pp. 593--604.

\bibitem[J4]{J3}
D. Jang: \emph{Circle actions on four-dimensional almost complex manifolds with discrete fixed point sets.} Int. Math. Res. Notices, Volume 2024, Issue 9, May 2024, Pages 7614--7639.

\bibitem[JM]{JM}
D. Jang and O. Musin: \emph{Circle actions on oriented 4-manifolds.} J. Topol. Anal., (2025), https://doi.org/10.1142/S1793525325500177.

\bibitem[Ka]{Ka}
Y. Karshon: \emph{Periodic Hamiltonian flows on four dimensional manifolds.} Mem. Amer.
Math. Soc. \textbf{672} (1999).

\bibitem[Kob]{K}
S. Kobayashi: \emph{Fixed points of isometries}. Nagoya Math. J. \textbf{13} (1958) 63--68.

\bibitem[Kos]{Ko}
C. Kosniowski: \emph{Fixed points and group actions.} Lecture Notes in Math. \textbf{1051} (1984) 603--609.

\bibitem[Ku]{Ku}
S. Kuroki: \emph{An Orlik-Raymond type classification of simply connected 6-dimensional torus manifolds with vanishing odd-degree cohomology.} Pac. J. Math. \textbf{280} (2016) 89--114.

\bibitem[L]{L}
H. Li: \emph{Revisit Hamiltonian $S^1$-manifolds of dimension 6 with 4 fixed points.} J. Geom. Phys. \textbf{213} (2025), 105489.

\bibitem[M]{M}
O. Musin: \emph{Actions of a circle on homotopy complex projective spaces.} Mat. Zametki \textbf{28} (1980) no. 1, 139--152 (Russian); Math. Notes \textbf{28} (1980) no. 1, 533--540 (English translation).

\bibitem[MO]{MO}
D. McGavran and H. S. Oh: \emph{Torus actions on 5- and 6-manifolds}. Indiana U. Math. J. \textbf{31} (1982) 363--376.

\bibitem[OR1]{OR1}
P. Orlik and F. Raymond: \emph{Actions of the Torus on 4-Manifolds. I.} Trans. Amer. Math. Soc. \textbf{152} (1970) 531--559.

\bibitem[OR2]{OR2}
P. Orlik and F. Raymond: \emph{Actions of the torus on 4-manifolds-II}. Topology \textbf{13} (1974) 89--112.

\bibitem[P]{P}
P. Pao: \emph{Non-linear circle actions on the 4-sphere and twisting spun knots.} Topology \textbf{17} (1978) no.3, 291--296.

\bibitem[R]{R}
F. Raymond: \emph{Classification of the actions of the circle on 3-manifolds.} Trans. Amer. Math. Soc. \textbf{131} (1968) 51--78.

\bibitem[T]{T}
S. Tolman: \emph{On a symplectic generalization of Petrie's conjecture.} Trans. Amer. Math. Soc. \textbf{362} (2010) no.8, 3963--3996.

\bibitem[W]{W}
M. Wiemeler: \emph{Rigidity of elliptic genera for non-spin manifolds}. Algebr. Geom. Topol. 25 (2025), no. 4, 2083--2097.

\end{thebibliography}
\end{document}